\documentclass[11pt]{elsarticle}
\usepackage{amsfonts,amssymb,amsmath,amsthm}
\usepackage{graphicx,color}
\usepackage[scale=0.77]{geometry}
\newtheorem{lem}{Lemma}[section]
\newtheorem{thm}[lem]{Theorem}
\newtheorem{cor}[lem]{Corollary}
\numberwithin{equation}{section}
\def\D{\mathrm D}
\def\R{\mathbb R}
\def\C{\mathbb C}
\def\d{\mathrm d\,}
\def\e{\mathrm e}
\DeclareMathOperator{\rank}{rank}

\DeclareMathOperator{\bdiag}{b-diag}
\DeclareMathOperator{\spec}{spec}
\DeclareMathOperator{\dist}{dist}
\DeclareMathOperator{\supp}{supp}
\begin{document}
\begin{frontmatter}
\title{Diffusion phenomena for partially dissipative hyperbolic systems}
\author[M]{Jens Wirth}\fnref{f1}
\address[M]{Mathematisches Institut, LMU M\"unchen, Theresienstr. 39/1, 80339 M\"unchen\\
Institut f\"ur Analysis, Dynamik und Modellierung, Universit\"at Stuttgart, Pfaffenwaldring 57, 70569 Stuttgart}
\ead{jens.wirth@mathematik.uni-stuttgart.de}
\fntext[f1]{The research presented in this paper was partially supported by the German Science Foundation (DFG) with grant WI 2064/5-1.}

\date{\today}
\begin{abstract}
We consider a partially dissipative hyperbolic system with time-dependent coefficients and show that under natural assumptions its solutions behave like solutions to a parabolic problem modulo terms of faster decay.

This generalises the well-known diffusion phenomenon for damped waves and gives some further insights into the structure of dissipative hyperbolic systems.
\end{abstract}

\end{frontmatter}

\section{Introduction}

The classical diffusion phenomenon observed by Hsiao--Liu \cite{Hsiao-Liu} and Nishihara \cite{Nishihara:1997} provides an asymptotic equivalence for solutions of the damped wave equation and corresponding solutions of the heat equation. We will recall the main results first before 
explaining our generalisations. If one considers the Cauchy problem for the damped wave equation
\begin{equation}
    u_{tt} - \Delta u + u_t = 0,\qquad u(0,\cdot)=u_0,\quad u_t(0,\cdot) = u_1
\end{equation}
on $\R^n$ together with the associated heat equation
\begin{equation}
   w_t = \Delta w,\qquad w(0,\cdot) = w_0 = u_0 + u_1,
\end{equation}
it is known that their solutions satisfy
\begin{equation}
    \| u(t,\cdot) - w(t,\cdot) \|_2 \le C t^{-1} \left( \|u_0\|_2 + \|u_1\|_2 \right)
\end{equation}
and corresponding estimates for higher order derivatives or in different $L^p$ norms. Note, that both solutions $u$ and $w$ do not decay in general. In this sense, the above estimate justifies to treat the solution of the heat equation as main term of the solution of the damped wave equation.

The above estimate is not the original one. There has been progress over the recent years to improve rates and constants, we refer only to a brief selection of papers. Nishihara \cite{Nishihara:2003} and Narazaki \cite{Narazaki:2004} extended the results to $L^p$--$L^q$ estimates,  Han--Milani \cite{Milani:2001} showed (independently) $L^1$-estimates for solutions. There have been extensions by Ikehata--Nishihara~\cite{Ikehata:2003c} and Chill--Hareaux \cite{Chill} to abstract evolution equations. Recently Radu--Todorova--Yordanov \cite{RTY:2011} used a different line of thought to estimate the difference of solutions in terms of the heat semigroup acting on the initial data. Their method allows to show
\begin{equation}
    \| u(t,\cdot) - w(t,\cdot) \|_2 \le C t^{-1} \left( \|\e^{t\Delta/2} u_0\|_2 + \|\e^{t \Delta/ 2}u_1\|_2 \right)
    + \e^{-t/16} \left( \|u_0\|_2 + \|u_1\|_{H^{-1}} \right)
\end{equation}
for $t\ge1$ and all initial data such that the norms on the right hand side are finite. Estimates of this form are valuable to the study of nonlinear damped wave equations and allow to show that critical exponents and associated global in time solutions in the supercritical cases are governed by the parabolic problem.

The author of this note extended the diffusion phenomenon to time-dependent dissipation terms in \cite{Wirth:2007}, see also \cite{Wirth:2010} for an overview of the results. There one considers
the wave equation
\begin{equation}\label{eq:CP-bdamped}
    u_{tt} - \Delta u + b(t) u_t = 0,\qquad u(0,\cdot)=u_0,\quad u_t(0,\cdot)=u_1
\end{equation}   
with (for simplicity monotone) smooth and positive coefficient function $b=b(t)$ subject to estimates
\begin{equation}
   \left| \partial_t^k b(t) \right| \le C_k b(t) \left(\frac1{1+t}\right)^k. 
\end{equation}
Such a dissipation term is called effective if $\lim_{t\to\infty} tb(t)=\infty$ and allows for a diffusion phenomen if in addition $1/b^3 \not\in \mathrm L^1(\R_+)$. Under this assumption solutions to the Cauchy problem \eqref{eq:CP-bdamped} behave asymptotically like solutions to the heat equation
\begin{equation}
    w_t  = \frac1{b(t)} \Delta w,\qquad w(0,\cdot) = w_0 = u_0 + \mu u_1,\qquad \mu = \int_0^\infty \exp\left(-\int_0^t b(s)\d s\right)\d t.
\end{equation}

This paper is {\em not} concerned with scalar second order equations. In the sequel we will consider hyperbolic systems with time-dependent coefficients. The situation without dissipation was treated by the author in \cite{RW:2008}, \cite{RW:2011} and allowed to derive dispersive estimates for solutions. On the other hand, a fully dissipative hyperbolic system will have exponential decay of its energy and is of less interest for us. We are going to study hyperbolic systems with a partial dissipation, where only one of the modes of the solution is damped directly while all the others decay due to coupling effects. Before formulating the problem we need to introduce some notation. 
We denote by $\mathcal T\{\ell\}$ the set of all smooth complex-valued functions $f\in\mathrm C^\infty([0,\infty);\C)$
subject to estimates
\begin{equation}
    \left| \partial_t^k f(t)\right| \le C_k \left(\frac1{1+t}\right)^{\ell+k}.
\end{equation}
We further denote by $\mathcal T\{-\infty\}=\bigcap_\ell \mathcal T\{\ell\}$ the set of rapidly decaying functions, it coincides with the Schwartz space $\mathcal S([0,\infty))$ on the half-line. Then we consider hyperbolic systems of the form 
\begin{equation}\label{eq:4:CP}
   \D_t U = \sum_{k=1}^n A_k(t)\D_{x_k} U + \mathrm i B(t) U,\qquad U(t,\cdot) = U_0\in\mathcal S(\R^n; \C^d),
\end{equation}
$\D_t = -\partial_t$ being the Fourier derivative.
Main difference to the considerations in \cite{RW:2008}, \cite{RW:2011} is that we do not assume
$B(t)$ to be of lower order in the $\mathcal T$-hierarchy. We assume that
\begin{equation}
   A_k(t), B(t) \in \mathcal T\{0\}\otimes\C^{d\times d} .
\end{equation}
Furthermore, we will make three main assumptions.
\begin{description}
\item[(B1)] The matrices $A_k(t)$  are self-adjoint and $\Re B(t)= \frac12 (B(t)+B(t)^*)\ge0$ is non-negative.
\item[(B2)] The matrices $B(t)$ are singular in the sense that $\det B(t)=0$ and
	\begin{equation*}
               \dist(0, \Re \spec B(t)\setminus\{0\}) \ge \kappa > 0
	\end{equation*}
	uniformly in $t\ge t_0$ for some positive constant $t_0$. The eigenvalue $0$ is assumed to be simple.
\item[(B3)] The matrices $A(t,\xi) = \sum_{k=1}^n A_k(t)\xi_k$ and $B(t)$ satisfy for all $\nu\in\mathbb C^d$
\begin{equation}
  \frac1c \| \nu\|^2 \le \sum_{j=0}^{d-1} \epsilon_j \| B(t) (A(t,\xi))^j \nu\|^2 \le c\|\nu\|^2,\qquad t\ge t_0,
\end{equation}
for any choice of numbers $\epsilon_0,\ldots,\epsilon_{d-1}>0$ and suitable constants $c$ and $t_0$ depending on them.
\end{description}
Assumption (B1) guarantees that the system is symmetric hyperbolic and (partially) dissipative.  Therefore,
the energy estimate 
\begin{equation} \label{eq:un-en-b}
        \|U(t,\cdot)\|_{L^2}\le \|U_0\|_{L^2}
\end{equation}
is valid and implies in particular existence and uniqueness of solutions.  By assumption (B2) we know that one mode is not dissipated, while assumption (B3) will be used to show that
 the high frequency parts of solutions are still exponentially decaying. Inspired by Beauchard--Zuazua \cite{BZ:2011}, we will refer to (B3) as uniform Kalman rank 
 condition. If $A_k$ and $B$ are independent of $t$ it just means that
 \begin{equation}
    \rank\big ( B \big  |  A(\xi) B \big | \cdots \big |  A(\xi)^{d-1}B \big) = d,
 \end{equation}
 which is the classical Kalman rank condition arising in the control theory of ordinary differential systems. Under certain natural assumptions, this is equivalent to 
 the algebraic condition of Kawashima--Shizuta \cite{KS:1985}, but the latter are more complicated to rewrite uniformly depending on parameters.

This paper is organised as follows. In section \ref{sec2} we will develop a block-diagonalisation scheme for small frequencies in order to conclude asymptotic information about solutions. It is based on assumptions (B1) and (B2). Later on in section~\ref{sec3} we will sketch the decay estimates for large frequencies based on assumption (B3) and Lyapunov functionals. Finally, section~\ref{sec4} gives decay estimates of solutions combined with comparison statements to corresponding parabolic problems.
 
\section{Diagonalisation for small frequencies}\label{sec2} 

\paragraph{Symbol classes}
The diagonalisation procedure is based on ideas from \cite{Jachmann:2010} and 
 \cite{Wirth:2009} and applies to the elliptic zone
\begin{equation}
  \mathcal Z_{\rm ell} (c) =\{(t,\xi) : |\xi| \le c, t\ge c^{-1}\}
\end{equation}
of the extended phase space $[0,\infty)\times \R^n$. It consists of small frequencies for large times. All considerations will be done for $t\ge t_0$ with $t_0$ sufficiently large. This ensures uniform block-diagonalisability of the matrix $B(t)$, i.e., due to assumption (B3) we find a matrix-valued function
\begin{equation}
   M(t) \in \mathcal T\{0\} \otimes\C^{d\times d}
\end{equation}
with uniformly bounded inverse such that
\begin{equation}
  M^{-1}(t) B(t) M(t) = \mathcal D(t) = \begin{pmatrix} 0 & \\ & \tilde{\mathcal D}(t) \end{pmatrix}  
\end{equation}
and $\tilde{\mathcal D}(t)$ has eigenvalues with imaginary parts in $[\kappa,\infty)$. As $B(t)$ is the main term for small frequencies, this allows to asymptotically decouple the system in a sufficiently small zone $\mathcal Z_{\rm ell} (c)$. For this we introduce the symbol classes
\begin{equation}
     \mathcal P\{m\} = \left\{ p(t,\xi) =  \sum_{|\alpha|\le m} p_{\alpha}(t) \xi^\alpha : p_\alpha (t) \in \mathcal T\{m-|\alpha|\} \right\}
\end{equation}
consisting of polynomials with coefficients in the $\mathcal T$-classes. Symbols from these classes can be thought of as homogeneous components, we
use 
\begin{equation}
    \mathcal P_\ge\{ m\} = \left\{ \sum_{k=m}^\infty p_k(t,\xi) : p_k(t,\xi) \in \mathcal P\{k\} \right\}
\end{equation}
to denote symbols which are (uniformly) convergent power series in $\mathcal Z_{\rm ell}(c)$ for $c$ sufficiently small. 

\paragraph{A first transformation}
Denoting $V^{(0)}(t,\xi) = M^{-1}(t) \widehat U(t,\xi)$, we obtain the equivalent system
\begin{equation}
   \D_t V^{(0)} = \left( \mathcal D(t) + \sum_{k=1}^n M^{-1}(t) A_k(t) M(t) \xi_k +\big( \D_t M^{-1}(t)\big) M(t) \right) V^{(0)}
\end{equation}
with $\mathcal D(t) = \bdiag_{(1,d-1)}(0, \tilde{\mathcal D}(t))$.  We write $\mathcal D(t)  + R_1(t,\xi)$ for its coefficient matrix. By construction, the matrix $R_1(t,\xi) \in\mathcal P\{1\}\otimes\C^{d\times d}$ is asymptotically of lower order.

\paragraph{The diagonalisation hierarchy} We start with the system
\begin{equation}
   \D_t V^{(0)} = \big(\mathcal D(t)  +  R_1(t,\xi) \big) V^{(0)} 
\end{equation}
with $R_1(t,\xi) \in \mathcal P\{1\}$. Before setting up the complete hierarchy, we will discuss its first step. We construct a matrix $N^{(1)}(t,\xi)\in{\mathcal P}\{1\}\otimes\C^{d\times d}$ such that
\begin{equation}
    \big(\D_t -\mathcal D(t) - R_1(t,\xi)\big) (\mathrm I + N^{(1)}(t,\xi))
    - (\mathrm I + N^{(1)}(t,\xi))  \big(\D_t -\mathcal D(t,\xi) - F_1(t,\xi)\big) \in{\mathcal P}\{2\}\otimes\C^{d\times d}
\end{equation}
holds true for some block-diagonal matrix $F_1(t,\xi)\in{\mathcal P}\{1\}\otimes \C^{d\times d}$. 
Collecting all terms not belonging to the right class yields again conditions for the
matrices $N^{(1)}(t,\xi)$ and $F_1(t,\xi)$. Indeed, 
\begin{equation}\label{eq:sylv}
  [\mathcal D(t) , N^{(1)}(t,\xi) ] = - R_1(t,\xi) + F_1(t,\xi) 
\end{equation}
must be satisfied and, therefore, we set 
\begin{equation}
   F_1(t,\xi) = \bdiag_{(1,d-1)} R_1(t,\xi),
\end{equation}
$\bdiag$ denoting the block-diagonal part to the partition of indices\footnote{For a partition $\Pi$ of $\{1,\ldots,d\}$ and a matrix $A$ we denote by $\bdiag_\Pi A$ the block-diagonal matrix with the same entries on the block diagonal, see \cite{Wirth:2009}. If partitions consist of sets of adjacent numbers, we denote them as tuple $\Pi=(d_1,\ldots,d_m)$, $d_1+\cdots+d_m=d$, encoding their sizes.}, combined with
\begin{equation}
   N^{(1)}(t,\xi) = \begin{pmatrix}0 & \tilde N_{1,1} (t,\xi)^\top \\ \tilde N_{1,2}(t,\xi)&0\end{pmatrix},
   \qquad 
   \tilde N_{1,j}(t,\xi) = \int_0^\infty \e^{-s \tilde{\mathcal D}(t) }\tilde R_{1,j}(t,\xi) \d s 
\end{equation}
for
\begin{equation}
   R_1(t,\xi)  - F_1(t,\xi) = \begin{pmatrix}0 &\tilde  R_{1,1} (t,\xi)^\top \\ \tilde R_{1,2}(t,\xi) &0 \end{pmatrix}.
 \end{equation}
Due to the spectral properties of $\tilde{\mathcal D}(t)$ the above integral converges and it is a matter of direct calculation to check that $N^{(1)}$ indeed satisfies the Sylvester equation \eqref{eq:sylv}. For further details see \cite{Wirth:2009}. As desired, the construction implies
$N^{(1)}(t,\xi), F_1(t,\xi)\in{\mathcal P}\{1\}\otimes \C^{d\times d}$. The construction generalises to remainders from better classes.

\begin{lem} Assume (B1) and (B2) are satisfied and let $k\in\mathbb N$, $k\ge 1$. 
Then there exists a constant $c_k$ and matrices $N_k(t,\xi)\in\mathcal P_\ge \{0\}\otimes \C^{d\times d}$, block-diagonal  matrices $F_{k}(t,\xi)\in\mathcal P_\ge \{1\}\otimes \C^{d\times d}$
and matrices $R_{k+1}(t,\xi)\in\mathcal P_\ge \{k+1\}\otimes \C^{d\times d}$ such that
\begin{equation}
   \big(\D_t- \mathcal D(t) - R_1(t,\xi)\big) N_k(t,\xi) = N_k(t,\xi)\big(\D_t -\mathcal D(t)-F_k(t,\xi)-R_{k+1}(t,\xi)\big) 
\end{equation} 
holds true within $\mathcal Z_{\rm ell}(c_k)$. Furthermore, $N_k(t,\xi)$ is uniformly invertible
within this zone.  
\end{lem}
\begin{proof}
We construct recursively the matrices $N^{(k)}(t,\xi)\in \mathcal P\{k\}\otimes \C^{d\times d}$ and $F^{(k)}(t,\xi)
\in \mathcal P\{k\}\otimes \C^{d\times d}$ block-diagonal, such that for
\begin{equation}
   N_K(t,\xi) = \mathrm I + \sum_{k=1}^K N^{(k)}(t,\xi) ,\qquad
   F_{K}(t,\xi) = \sum_{k=1}^{K} F^{(k)}(t,\xi),
\end{equation}
the estimate
\begin{equation}
    B_K(t,\xi) = \big(\D_t -\mathcal D(t) - R_1(t,\xi)\big) N_K(t,\xi)
    -N_K(t,\xi)  \big(\D_t -\mathcal D(t) - F_K(t,\xi)\big) \in \mathcal P_\ge \{K+1\}
\end{equation}
is valid. We just did this for $K=1$, it remains to do the recursion $k\mapsto k+1$. Assume
$B_k(t,\xi)\in\mathcal P\{k+1\}\otimes \C^{d\times d}$. The requirement to be met is that
\begin{equation}
    B_{k+1}(t,\xi) - B_k(t,\xi) =  -[\mathcal D(t), N^{(k+1)}(t,\xi)] + F^{(k+1)}(t,\xi)  \mod \mathcal P_\ge \{k+2\} \otimes \C^{d\times d}
\end{equation}
for block-diagonal $F^{(k+1)}(t,\xi)$, which yields 
\begin{equation} 
F^{(k+1)}(t,\xi) = -\bdiag_{(1,d-1)} B_k(t,\xi)
\end{equation}
 together with
\begin{equation}
    N^{(k+1)}(t,\xi) = \begin{pmatrix}0 & \tilde N_{k+1,1} (t,\xi)^\top \\ \tilde N_{k+1,2}(t,\xi)&0\end{pmatrix},
   \qquad 
   \tilde N_{k+1,j}(t,\xi) =-  \int_0^\infty \e^{-s \tilde{\mathcal D}(t) }\tilde B_{k,j}(t,\xi) \d s 
\end{equation}
for the components
\begin{equation}
   B_k(t,\xi)  + F_1(t,\xi) = \begin{pmatrix}0 &\tilde  B_{k,1} (t,\xi)^\top \\ \tilde B_{k,2}(t,\xi)&0\end{pmatrix}.
 \end{equation}
It is evident that the construction implies $F^{(k+1)}(t,\xi),N^{(k+1)}(t,\xi)\in\mathcal P\{k+1\}\otimes \C^{d\times d}$
together with $B_{k+1}(t,\xi) \in \mathcal P_\ge \{k+2\}\otimes \C^{d\times d}$.

The matrices $N_k(t,\xi)\in\mathcal P_\ge \{0\}\otimes \C^{d\times d}$ are invertible with inverse $N^{-1}_k(t,\xi)\in\mathcal P_\ge \{0\}\otimes \C^{d\times d}$ if we restrict our consideration to a sufficiently small elliptic zone $\mathcal Z_{\rm ell}(c_k)$. The latter follows by using the Neumann series for the inverse.
The result of the above consideration can be summarised in the following lemma.
\end{proof}

It is worth having a closer look at the upper-left corner entry of the matrix $F_k(t,\xi)$ and consequences for them based on assumption (B2).
The above lemma implies that modulo $\mathcal P_\ge \{3\}$ the entry is of the form
\begin{equation}\label{eq:4:f1k}
  f_1^{(k)}(t,\xi) = \mathrm i \sum_{i,j=1}^d \alpha_{i,j}(t) \xi_i\xi_j + \sum_{i=1}^d \beta_i(t) \xi_i +  \gamma(t)
   \mod \mathcal P_\ge \{3\}
\end{equation}
with $\alpha_{i,j}(t)\in \mathcal T\{0\}$, $\beta_i(t)\in\mathcal T\{0\}$ and $\gamma(t)\in\mathcal T\{1\}$. On the other hand, modulo $\mathcal O(t^{-1})$
the eigenvalues of $F_k(t,\xi)$ and of $A(t,\xi)$ coincide. As (B2) is a spectral condition implying that $0$ is a local (quadratic) minimum of an eigenvalue branch contained in the complex upper half-plane, 
some terms in \eqref{eq:4:f1k} have to vanish. In particular we see that $\beta_i(t)$ has to decay, $\beta_i(t)\in\mathcal T\{1\}$, and also that the real part of the quadratic matrix $(\alpha_{i,j}(t))_{i,j}$ is {positive definite} modulo $\mathcal T\{1\}$. The latter is a direct consequence of the non-degeneracy of that minimum.

\begin{cor}\label{cor:sysDP:par-terms} Assume (B1) and (B2) and let $k\ge 1$. Then modulo $\mathcal P_\ge \{3\}$ the upper-left  corner entry of $F_k(t,\xi)$ satisfies
\begin{equation}\label{eq:2-21}
  f_1^{(k)}(t,\xi) = \mathrm i \xi^\top {\boldsymbol\alpha(t)} \xi + \boldsymbol\beta(t)^\top \xi + \gamma(t) 
   \mod \mathcal P_\ge \{3\}
\end{equation}
with $\boldsymbol\alpha(t)\in\mathcal T\{0\}\otimes\mathbb C^{d\times d}$ having positive definite real part uniform in $t\ge t_0$ for $t_0$ sufficiently large, $\boldsymbol\beta(t)\in\mathcal T\{1\}\otimes \mathbb C^d$ and $\gamma(t)\in\mathcal T\{1\}$. 
\end{cor} 

If the matrix $B(t)$ is self-adjoint for all $t$ then we can choose $M(t)$ unitary and, therefore, the construction gives that $\gamma(t)$
is real-valued modulo $\mathcal T\{2\}$. By assumption (B1) in the form of estimate \eqref{eq:un-en-b} we know that $\exp\left(-\Im \int_{t_0}^t \gamma(\theta)\d\theta\right)\lesssim1$. 

\section{Asymptotic integration and small frequency expansions}
\paragraph{Fundamental solutions} We consider the transformed problem in $V^{(k)}(t,\xi) = N_k(t,\xi) V^{(0)}(t,\xi)$, 
\begin{equation}\label{eq:sysDP:eq:CPk}
  \mathrm D_t V^{(k)}(t,\xi) = \big( \mathcal D(t) + F_k(t,\xi) + R_{k+1} (t,\xi) \big) V^{(k)}(t,\xi), 
\end{equation}
and reformulate this as integral equation for its fundamental solution
$\mathcal E_k(t,s,\xi)$, i.e.,  the matrix-valued solution to initial data $\mathcal E_k(s,s,\xi) = \mathrm I\in\mathbb C^{d\times d}$. 
Let for this $\Theta_k(t,s,\xi)$ be the fundamental solution to the block-diagonal system $\mathrm D_t - \mathcal D(t)-F_k(t,\xi)$. Then 
\begin{equation}
   \Theta_k(t,s,\xi) = \begin{pmatrix} \Xi_k(t,s,\xi) & 0 \\ 0 & \tilde\Theta_k(t,s,\xi) \end{pmatrix},
\qquad    \| \tilde\Theta_k(t,s,\xi) \| \lesssim \mathrm e^{-\tilde c ({t-s})},
\end{equation}
for $t\ge s$ uniformly within $\mathcal Z_{\rm ell}(\epsilon)$ for $\epsilon\le c_k$ sufficiently small.
Here, 
\begin{equation}
 \Xi_k(t,s,\xi) = \exp\left(\mathrm i \int_s^t f_1^{(k)}(\theta,\xi) \mathrm d\theta\right)
\end{equation}
is uniformly bounded and gives (for $k=2$) the fundamental solution to a parabolic problem
and $\tilde \Theta_k(t,s,\xi)$ is exponentially decaying as the fundamental solution of a dissipative system.
Furthermore, the matrix $\mathcal E_k(t,s,\xi)$ satisfies the Volterra integral equation
\begin{equation}
   \mathcal E_k(t,s,\xi)  = \Theta_k(t,s,\xi) + \int_s^t \Theta_k(t,\theta,\xi) R_{k+1}(\theta,\xi) 
   \mathcal E_k(\theta,s,\xi) \mathrm d\theta.
\end{equation}
We solve this equation using the Neumann series
\begin{multline}\label{sysDP:eq:NeumannS}
  \mathcal E_k(t,s,\xi) = \Theta_k(t,s,\xi) + \sum_{\ell=1}^\infty \mathrm i^\ell
  \int_s^t \Theta_k(t,t_1,\xi) R_{k+1}(t_1,\xi)\int_s^{t_1} \cdots\\\cdots \int_s^{t_{\ell-1}} 
   \Theta_k(t_{\ell-1},t_\ell,\xi) R_{k+1}(t_\ell,\xi)\mathrm d t_\ell \cdots \mathrm d t_1.
\end{multline}
This series converges and and its value can be estimated by 
\begin{equation}
    \| \mathcal E_k(t,s,\xi) \| \le \exp\left(\int_s^t \| R_{k+1}(\theta,\xi) \|\mathrm d\theta\right).
\end{equation}
Based on the estimates for the remainder term 
$R_{k+1}(t,\xi) \in\mathcal P_\ge \{k+1\}\otimes \C^{d\times d}$, we even obtain uniform convergence 
within the smaller zone $\mathcal Z_{\rm ell}(c_k)\cap \{ t|\xi|^{(k+1)/2} \le \delta \}$
for any constant $\delta$  and $\| \mathcal E_k(t,s,\xi) - \Theta_k(t,s,\xi) \| \to 0$ as
$c_k\to0$ for fixed $\delta>0$ as soon as we choose $k\ge 1$. 

In particular we obtain a pointwise bound in terms of $\Xi_k(t,s,\xi)$ for $t|\xi|^2\le \delta$.

\begin{lem}\label{lem:sysDP:sol_est-1} Assume (B1), (B2) and
   let $k\ge 2$ together with $\delta>0$. Then the fundamental solution $\mathcal E_k(t,s,\xi)$ 
   satisfies the uniform bound
   \begin{equation}
      \| \mathcal E_k(t,s,\xi) \| \le C_k  | \Xi_k(t,s,\xi) |
      ,\qquad  t\ge s\ge t_0, \quad t|\xi|^2\le \delta,
   \end{equation}
   for some constant $C_k>0$ depending on $t_0$, $\delta$ and $k$.   
\end{lem}
\begin{proof}
As $\Xi_k^{-1}(t,s,\xi)$ is scalar, we conclude that the matrix $\Xi_k^{-1}(t,s,\xi)\mathcal E_k(t,s,\xi)$ satisfies \eqref{eq:sysDP:eq:CPk} 
with $F_k(t,\xi)$ replaced by $F_k(t,\xi)-f_1^{(k)}(t,\xi)\mathrm I$ and, therefore, \eqref{sysDP:eq:NeumannS} 
with $\Theta_k$ replaced by $\Xi_k^{-1}\Theta_k$ in all places it occurs. The matrices $\Xi_k^{-1}\Theta_k$ are uniformly bounded and, hence,
 \begin{align*}
  \| \mathcal E_k(t,s,\xi) \|& \le | \Xi_k(t,s,\xi) | \exp\left(C' \int_s^t \|R_{k+1}(\theta,\xi)\|\mathrm d\theta \right).
\end{align*}
Furthermore, for $k\ge 2$ the remaining integral is uniformly 
bounded on this set.
\end{proof}

\section{Lyapunov functionals and parabolic type estimates}\label{sec3}
In this section we will partly follow the considerations of Beauchard--Zuazua, \cite{BZ:2011}, and explain how condition (B3) of Section~\ref{sec2} allows to derive parabolic type decay estimates for solutions to the Cauchy problem \eqref{eq:4:CP}. The construction in \cite{BZ:2011} was inspired by the Lyapunov functionals used by Villani \cite{Villani:2011}.

\begin{lem}\label{lem:sysDP:KalmEst}
Assume (B1), (B2), (B3). Then all solutions to \eqref{eq:4:CP} to Schwartz initial data satisfy the point-wise estimate
\begin{equation}
   \| \widehat U(t,\xi) \|^2 \le C \mathrm e^{-\gamma t [\xi]^2 } \|\widehat U_0(\xi)\|^2,\qquad [\xi] = |\xi|/\langle \xi\rangle \simeq \min\{|\xi|,1\},
\end{equation}
in Fourier space with constants $C$ and $\gamma$ depending only on the coefficient matrices $A_k(t)$ and $B(t)$.
\end{lem}
\begin{proof}[Sketch of proof]
The proof follows essentially \cite[Sec. 2.2]{BZ:2011}, we will only explain the major steps and necessary modifications to incorporate the time-dependence of 
matrices. As the problem is $L^2$-well-posed and dissipative, it suffices to prove the statement only for $t\ge t_0$ with a sufficiently large $t_0$.

For a still to be specified selection $\epsilon=(\epsilon_0,\ldots, \epsilon_{d-1})$ of positive reals, $\epsilon_j>0$, we consider the Lyapunov functional 
\begin{multline}
   \mathbb L_\epsilon[\widehat U;t,\xi] = \| \widehat U(t,\xi) \|^2 +\\+ \min\{|\xi|,|\xi|^{-1}\} \, \sum_{j=1}^{d-1} \epsilon_j \Im \langle B(t) A(t,{\xi}/{|\xi|}) ^{j-1} \widehat U(t,\xi) , B(t) A(t,{\xi}/{|\xi|})^{j} \widehat U(t,\xi) \rangle.
\end{multline}
The uniform Kalman rank condition (B3) implies that for suitable choices of the parameters $\epsilon$ and for $t\ge t_0$ the two-sided estimate
\begin{equation}
  \frac14 \|\widehat U(t,\xi)\|^2  \le  \mathbb L_\epsilon[\widehat U;t,\xi] \le 4 \|\widehat U(t,\xi)\|^2
\end{equation}
holds true. Therefore, all we have to do is to prove the desired estimate for $ \mathbb L_\epsilon[\widehat U;t,\xi] $, which follows from the differential inequality
\begin{equation}
   \partial_t  \mathbb L_\epsilon[\widehat U;t,\xi] + \gamma [\xi]^2  \mathbb L_\epsilon[\widehat U;t,\xi] \le 0
\end{equation}
by an application of Gronwall's inequality. Our aim is to find a suitable $\gamma$ and a suitable family $\epsilon$ for this differential inequality to be true. 

Formally differentiating $\mathbb L_\epsilon[\widehat U;t,\xi]$ with respect to $t$ yields the terms considered by \cite[Sec. 2.2]{BZ:2011} together with further
terms containing derivatives of the coefficient matrices. The first ones are estimated exactly like in that paper, while the latter ones are bounded by 
\begin{equation}\label{eq:BZ-rem-terms}
 \mathcal O(t^{-1} ) \min\{|\xi|,|\xi|^{-1}\}  \|\widehat U(t,\xi)\|^2
\end{equation}
due to our assumptions. As \eqref{eq:BZ-rem-terms} is dominated by  $\gamma [\xi]^2  \mathbb L_\epsilon[\widehat U;t,\xi]$ whenever $t|\xi|\gtrsim1$, the desired bound follows on this zone by choosing $\gamma$ small.
It remains to consider $t|\xi|\lesssim1$. Here $A(t,\xi)$ can be treated as small perturbation of $B(t)$ and the diagonalisation scheme and Lemma~\ref{lem:sysDP:sol_est-1} yield already the corresponding bound.
\end{proof}

\section{Estimates for solutions to partially dissipative hyperbolic systems}\label{sec4}

\paragraph{$L^p$--$L^q$ estimates} First we will conclude parabolic type $L^p$--$L^q$ decay estimates. They are a direct consequence of Lemma~\ref{lem:sysDP:KalmEst} in combination 
with H\"older's inequality and the boundedness properties of Fourier transform.

\begin{thm}
Assume (B1), (B2), (B3). Then all solutions to  \eqref{eq:4:CP} satisfy 
\begin{equation}
   \| U(t,\cdot)\|_{L^q} \le C (1+t)^{-\frac n2(\frac1p-\frac1q)} \| U_0\|_{W^{p,r}}
\end{equation}
for all $1\le p\le 2\le q\le \infty$ and with $r\ge n(1/p-1/q)$.
\end{thm}
\begin{proof}
For $|\xi|\gtrsim 1$ Lemma~\ref{lem:sysDP:KalmEst}  in combination with Sobolev embedding theorem yields exponential decay under the imposed regularity. Therefore, it is enough to
consider bounded $\xi$. But then the estimate follows from Lemma~\ref{lem:sysDP:KalmEst}
\begin{align*}
   \| U(t,\cdot) \|_{L^q} \le \|\widehat U(t,\cdot)\|_{L^{q'}} \le \| \mathrm e^{-\gamma t|\cdot|^2} \|_{L^r} \|\widehat U_0\|_{L^{p'}} \le C (1+t)^{-\frac n2(\frac1p-\frac1q)} \|U_0\|_{L^p} 
\end{align*}
provided $\supp \widehat U_0(\xi) \subset \{ \xi : |\xi|\le 1\}$ and $\frac1 r=\frac1{q'}-\frac1{p'}=
\frac1p-\frac1q$.
\end{proof}

Note, that this a typical parabolic decay estimate. It highlights an underlying diffusive structure. There are related improved decay estimates in this situation, e.g., by assuming moment  and decay conditions on the data similar to \cite{Ikehata:2003b}, \cite{Ikehata:2003e}.

\paragraph{Diffusion phenomena}
Now we will combine the estimates of the previous section with slightly improved results 
obtained from the low-frequency diagonalisation. First we construct a parabolic reference 
problem, whose fundamental solution is given by $\Xi_2(t,s,\xi)$ and afterwards we will explain why and in what sense solutions are asymptotically equivalent.

Following Corollary~\ref{cor:sysDP:par-terms} it is reasonable to consider the parabolic problem\begin{equation}\label{sysDP:eq:par} 
   \partial_t  w = \nabla \cdot \boldsymbol \alpha(t) \nabla w + \boldsymbol\beta (t) \cdot\nabla w + 
 \mathrm i  \gamma(t) w,\qquad w(t_0)=w_0,
\end{equation}
for a scalar-valued unknown function $w_0$. To relate both problems, we observe that
the first row of $\mathcal E_k(t,s,0)$ tends to a limit as $t\to\infty$.  This is just a consequence of the integrability of $R_{k+1}(t,0)$ for $k\ge2$. We use this to define
\begin{equation}
   W_k(s) = \lim_{t\to\infty} e_1^\top \mathcal E_k(t,s,0).
\end{equation}
It is easy to see that $W_k(s)=W_2(s)$ for all $k$.

\begin{lem}\label{sysDP:lem:solEst-2} 
The fundamental solution $\mathcal E_k(t,s,\xi)$, $k$ sufficiently large, satisfies the estimate
\begin{equation}
   \| \mathcal E_k(t,s,\xi) - \Xi_k(t,s,\xi) e_1 W_2(s) \| \le C_k (1+t)^{-\frac12},\qquad t\ge s\ge t_0,
\end{equation}
uniformly on $|\xi|\le 1$.
\end{lem}
\begin{proof}
We make use of a constant $\delta>0$, to be fixed later on, to decompose the extended phase space into several zones. 

\paragraph{Part 1} If $t|\xi|^{2} \ge \delta \log t$ with $\delta$ chosen large enough, both terms 
can be estimated separately by $\exp(-\tilde\gamma t|\xi|^2)$ for some constant $\gamma$.
This follows for the first one by  Lemma~ \ref{lem:sysDP:KalmEst} and for the second
by the parabolicity of \eqref{sysDP:eq:par}  in consequence of 
Corollary~\ref{cor:sysDP:par-terms}.  But then
\begin{equation}
\mathrm e^{-\tilde\gamma t|\xi|^2} \le \mathrm e^{-\tilde \gamma\delta \log t} = t^{-\tilde \gamma\delta} \lesssim t^{-\frac12},\qquad \tilde\gamma\delta\ge\frac12.
\end{equation} 

\paragraph{Part 2}
If  $t|\xi|^2\le \delta$ for some $\delta$, we use the results from the asymptotic integration of the diagonalised system. First, we claim that $\Xi_k^{-1}(t,t_0,\xi)e_1^\top\mathcal E_k(t,t_0,\xi)$
converges locally uniform in $\xi$ as $t\to t_\xi$ for $t_\xi |\xi|^2 = \delta$. To see this, we use
the Neumann series representation of $\Xi_k^{-1}(t,t_0,\xi)\mathcal E_k(t,t_0,\xi)$ (i.e.,  \eqref{sysDP:eq:NeumannS} with $\Theta_k$ replaced by $ \Xi_k^{-1}\Theta_k $) multiplied by $e_1^\top$
\begin{align}
  \Xi_k^{-1}&(t,t_0,\xi) e_1^\top\mathcal E_k(t,t_0,\xi)  = e_1^\top + \sum_{\ell=1}^\infty \mathrm i^\ell \int_{t_0}^t  e_1^\top R_{k+1}(t_1,\xi)  \int_{t_0}^{t_1} \Xi_k^{-1}(t_1,t_2,\xi) \Theta_k (t_1,t_2,\xi) R_{k+1}(t_2,\xi) \notag\\& \qquad\qquad \times \int_{t_0}^{t_2}\cdots \int_{t_0}^{t_{\ell-1}}  \Xi_k^{-1}(t_{\ell-1},t_\ell,\xi)  \Theta_k(t_{\ell-1},t_\ell,\xi) R_{k+1}(t_\ell,\xi) \d t_\ell\cdots \d t_2 \d t_1
\end{align}
and use the Cauchy criterion.
We denote the resulting limit as $W_k(s,\xi)$ and observe that it coincides with $W_k(s)$ for $\xi=0$. Next, we show the estimates
\begin{equation}\label{eq:5.6}
   \| W_k(s,\xi) - W_k(s) \| \lesssim |\xi|,
\end{equation}
and
\begin{equation}
  \| e_1^\top \mathcal E_k(t,s,\xi) - \Xi_k(t,s,\xi) W_k(s,\xi) \| \lesssim t^{-\frac12}.
\end{equation}
The first of these estimates follows as uniform limit for estimates of the difference
\begin{equation}
\Xi_k^{-1}(t,s,0)e_1^\top\mathcal E_k(t,s,0)-\Xi_k^{-1}(t,s,\xi)e_1^\top\mathcal E_k(t,s,\xi).
\end{equation} 
Indeed, using again the Neumann series we see that the first terms are equal. The second terms
are reduced to the estimate
\begin{equation}
\|R_{k+1}(t_1,\xi)-R_{k+1}(t_1,0)\| \lesssim |\xi| (t_1^{-k} + |\xi|^k),\qquad t_1\ge t_0,
\end{equation}
following directly from the definition of the $\mathcal P\{k+1\}$-classes. 
Therefore, 
\begin{equation}
   \|\Xi_k^{-1}(t,s,0)e_1^\top\mathcal E_k(t,s,0)-\Xi_k^{-1}(t,s,\xi)e_1^\top\mathcal E_k(t,s,\xi)\|\lesssim |\xi| \int_s^t (\theta^{-k} + |\xi|^k) \mathrm d\theta  + 
   \mathrm{l.o.t.}
\end{equation}
and the right-hand side is uniformly bounded by $|\xi|$. Taking limits proves the estimate.
The second estimate is similar. Again using the Neumann series we see that this difference
can be estimated by
\begin{multline}
   \| \Xi_k^{-1}(t,s,\xi)  e_1^\top \mathcal E_k(t,s,\xi) - W_k(s,\xi) \| 
 \\  \lesssim \int_t^{t_\xi} \| R_{k+1}(\tau,\xi)\| \exp\left(\int_{t_0}^{t_\xi} \|R_{k+1}(\theta,\xi) \|\mathrm d\theta
   \right)\mathrm d\tau \lesssim t^{-\frac12} 
\end{multline}
due to $\|R_{k+1}(t,\xi)\|=\mathcal O(t^{-\frac32})$ for $k\ge 2$ and $t|\xi|^2\le \delta$.

Combining both of the above estimates and using that the other rows in $\mathcal E_k$ are
exponentially decaying we get
\begin{align}
    \|  \mathcal E_k(t,s,\xi) - \Xi_k(t,s,\xi) e_1 W_k(s) \|
     & \le   \|  \mathcal E_k(t,s,\xi) - \Xi_k(t,s,\xi) e_1 W_k(s,\xi) \| \notag\\&\quad+ \|  \Xi_k(t,s,\xi) e_1 W_k(s,\xi)- \Xi_k(t,s,\xi) e_1 W_k(s)\| \notag\\& \lesssim t^{-\frac12} + \mathrm e^{-\tilde \gamma t |\xi|^2}|\xi| 
     \lesssim t^{-\frac12}.
\end{align}

\paragraph{Part 3} 
It remains to consider the logarithmic gap between both parts, i.e., $\delta \le t|\xi|^2 \le \delta\log t$. Here we use that for $k$ sufficiently large the remainder term $R_{k+1}(t,\xi)$ decays as
$t^{-k-1+\epsilon}$, while the polynomial growth rate of $\Xi^{-1}_k(t,\xi)$ is independent of 
$k$ for large $k$. Choosing $k$ large enough, the Neumann series argument gives
\begin{equation}
  \| \Xi_k^{-1}(t,s,\xi) e_1^\top \mathcal E_k(t,s,\xi) - \tilde W_k(s,\xi) \| 
  \lesssim \int_t^{\tilde t_\xi} \| \Xi_k^{-1}(\theta,s,\xi) R_{k+1}(\theta,\xi) \| \mathrm d\theta  
  \lesssim t^{-\frac12}
\end{equation}
with $\tilde t_\xi$ defined by $\tilde t_\xi |\xi|^2 = \delta \log \tilde t_\xi$ and 
\begin{equation}
\tilde W_k(s,\xi) = \lim_{t\to\tilde t_\xi} \Xi_k^{-1} (t,s,\xi) e_1^\top \mathcal E_k(t,s,\xi).
\end{equation}
The
existence of the latter limit follows for large $k$ and again $\tilde W_k(s,\xi) - W_k(s)$
coincides up to order $\mathcal O(|\xi|)$.
\end{proof}

To obtain a statement in terms of the original equation, we introduce
$K(t,\xi) = M(t)N_2(t,\xi) e_1$. By definition we have 
$K(t,\xi) \in( \mathcal P \{0\}+\mathcal P \{1\}+ \mathcal P\{2\})\otimes\C^{d}$. We define further $w_0 = W U_0$ in such a way that we cancel 
the main term of the solution within $\mathcal Z_{\rm ell}(c_k)\cap \{ t|\xi| \le \delta \}$, i.e., 
we define
\begin{equation}\label{eq:4:init-data-DP}
   \widehat w_0 = W_2(t_0) N_2^{-1}(t_0,\xi) M^{-1}(t_0) \mathcal E(t_0,0,\xi) \chi(\xi) \widehat U_0
\end{equation}
with $\chi(\xi)\in C^\infty_0(\mathbb R^n)$, $\chi(\xi)=1$ near $\xi=0$ and $\mathrm{supp}\,\chi\subset B_{c_2}(0)$. Then the estimate of Lemma~\ref{sysDP:lem:solEst-2} implies the following statement. The logarithmic term is caused by comparing $\Xi_k(t,s,\xi)$ with $\Xi_2(t,s,\xi)$.

\begin{thm} Assume (B1), (B2), (B3) and let $U(t,x)$ be solution to \eqref{eq:4:CP}. The the solution $w(t,x)$ to \eqref{sysDP:eq:par} with data given by \eqref{eq:4:init-data-DP} satisfies
\begin{equation}\label{eq:diff-phen}
\| U(t,\cdot) -  K(t,\D) 
w(t,\cdot) \|_{L^2}   \le 
    C' (1+t)^{-\frac12} \log(\mathrm e+ t) \|U_0\|_{L^2}.
\end{equation}
\end{thm}

\section{Concluding remarks}
\paragraph{(1)}The logarithm in \eqref{eq:diff-phen} is most likely not sharp, while we conjecture that the rate $(1+t)^{-1/2}$ is sharp as long as no further symmetry conditions are satisfied. 

\paragraph{(2)} Assumption (B1) may be strengthend by assuming that $B(t)\ge 0$. Under this assumption the problem has an additional symmetry in Fourier space,
\begin{equation}
   ( A(t,\xi) + \mathrm i B(t) )^* =  A(t,-\xi) - \mathrm i B(t) = - (A(t,\xi)+\mathrm i B(t)) 
\end{equation}
and it follows that the fundamental solution satisfies $\mathcal E(t,s,-\xi) = \mathcal E(t,s,\xi)^*$. In consequence, the coefficient ${\boldsymbol\beta}(t)$ of the associated parabolic problem is purely imaginary, the estimate in \eqref{eq:5.6} is of order $|\xi|^2$ and similar for the following estimates. Finally, the estimate \eqref{eq:diff-phen} improves by half an order.

\end{document}